\documentclass[11pt]{article}
\usepackage{amsmath}
\usepackage{amssymb}
\usepackage{amsthm}
\usepackage[mathscr]{eucal}
\usepackage{lineno}
\theoremstyle{definition}
\newtheorem{definition}{Definition}%[section]
\newtheorem{theorem}[definition]{Theorem}
\newtheorem*{theorem*}{Conjecture}
\newtheorem{proposition}[definition]{Proposition}
\newtheorem{lemma}[definition]{Lemma}

\theoremstyle{remark}
\newtheorem{remark}[definition]{Remark}

\newcounter{enumctr}

\newcommand{\R}{\mathbb{R}}

\newcommand{\C}{\mathbb{C}}
\newcommand{\id}{\hbox{id}}

\newcommand{\eps}{\varepsilon}
\renewcommand{\phi}{\varphi}

\newcommand{\rT}{\mathrm {T}}
\providecommand{\keywords}[1]{\textbf{\textbf{Key words: }} #1}
\begin{document}

%\linenumbers

\title{\vspace*{-10mm}
An instability theorem for nonlinear\\ fractional differential systems}
\author{
N.D.\ Cong\thanks{Email: ndcong@math.ac.vn,  Institute of Mathematics, Vietnam Academy of Science and Technology, Viet Nam.}, 
T.S.\ Doan\thanks{Email: dtson@math.ac.vn,  Institute of Mathematics, Vietnam Academy of Science and Technology, Viet Nam and Department of Mathematics, Hokkaido University, Japan.}, 
S.\ Siegmund\thanks{Email: stefan.siegmund@tu-dresden.de, Department of Mathematics, Technische Universit\"{a}t Dresden, Dresden, Germany.} 
{ }and H.T.\ Tuan\thanks{Email: httuan@math.ac.vn,  Institute of Mathematics, Vietnam Academy of Science and Technology, Viet Nam.}
}
\date{}
\maketitle
\begin{abstract}
In this paper, we give a criterion on instability of an equilibrium of a nonlinear Caputo fractional differential system. More precisely, we prove that if the spectrum of the linearization has at least one eigenvalue in the sector $$\left\{\lambda\in\C\setminus\{0\}:|\arg{(\lambda)}|<\frac{\alpha \pi}{2}\right\},$$ where $\alpha\in (0,1)$ is the order of the fractional differential system, then the equilibrium of the nonlinear system is unstable.
\end{abstract}
%
%
%\linenumbers
\section{Introduction}
In recent years, fractional differential equations have attracted increasing interest due to their many applications in various fields of science and engineering, see e.g., \cite{Oldham,Samko}.
One of the fundamental problems of the qualitative theory of fractional differential equations is stability theory. So far, there have been a number of publications on stability theory for different types of fractional systems, e.g.,\  linear fractional differential equations \cite{Matignon_1996, Bonilla,Cermak13, Kaminski}, linear fractional difference equations \cite{Abu,Cermak} and nonlinear fractional differential equations \cite{Ahmed, Deng, Agarwal, Cong_3}.

In this paper, we are interested in stability of the trivial solution of a
nonlinear Caputo fractional differential system of order $\alpha$, $0<\alpha<1$,
\begin{equation}\label{IntroEq}
^{C\!}D_{0+}^\alpha x(t)=Ax(t)+f(x(t)),
\end{equation}
where $t\geq 0$, $x\in \R^d$, $A\in\R^{d\times d}$ and $f:\R^d \rightarrow \R^d$ is continuous on $\R^d$ and Lipschitz continuous in a neighborhood of the origin, $f(0)=0$ and $\lim_{r\to 0} \ell_f(r)=0$, where
\begin{equation}\label{eqn.l_f}
\ell_f(r):=\sup_{\substack{x,y\in B_{\R^d}(0,r)\\ x\neq y}}\frac{\|f(x)-f(y)\|}{\|x-y\|},
\end{equation}
with $B_{\R^d}(0,r):=\big\{x\in\R^d: \|x\|\leq r\big\}$. Similarly as for ordinary differential equations, we would expect for fractional differential equations that if the linear system
\begin{equation}\label{LinearizationEq}
^{C\!}D_{0+}^\alpha x(t)=Ax(t),
\end{equation}
 is asymptotically stable (or unstable) then the trivial solution of the perturbed system \eqref{IntroEq} is also asymptotically stable (or unstable, respectively), since these conclusions hold in the theory of ordinary differential equations, see e.g., \cite[Chapter 13]{Coddington}. In \cite{Cong_3}, we give an affirmative answer in the case when the linear system \eqref{LinearizationEq} is asymptotically stable by proving that the trivial solution of system \eqref{IntroEq}  is then also asymptotically stable.

However, in the case that system \eqref{LinearizationEq} is unstable, e.g.,\ if $A$ has at least one eigenvalue $\lambda$ with its argument\footnote{For a nonzero complex number $\lambda$, we define its argument to be in the interval $-\pi < \arg{(\lambda)}\leq \pi$.} satisfying that $|\arg(\lambda)|<\frac{\alpha \pi}{2}$, the question whether the trivial solution of \eqref{IntroEq} is unstable still remains open. Notice that in the scalar case, this question is part of a conjecture of J.~Audounet, D.~Matignon and G.~Montseny~\cite[Theorem 3, p.~81]{Matignon} which is stated as follows.
\begin{theorem*}
The local stability of the equilibrium $x^{*}=0$ of the nonlinear fractional differential system $^{C\!}D_{0+}^{\alpha} x(t) = f(x(t))$ is governed by the global stability of the linearized system near the equilibrium $^{C\!}D_{0+}^{\alpha} x(t) = \lambda x(t)$, where $\lambda = f'(0)\in \C$, namely:
\begin{itemize}
\item [(i)] $x^{*}=0$ is locally asymptotically stable if $|\arg(\lambda)|>\frac{\alpha \pi}{2}$,
\item [(ii)] $x^{*}=0$ is not locally stable if $|\arg(\lambda)|<\frac{\alpha \pi}{2}$.
\end{itemize}
\end{theorem*}
In this paper, we establish an instability theorem for the nonlinear Caputo fractional differential system \eqref{IntroEq} when the linearization \eqref{LinearizationEq} is unstable. As a consequence, we also prove statement (ii) of the conjecture above. The main ingredient in the proof is to construct a suitable Lyapunov--Perron operator and to show that a bounded solution of \eqref{IntroEq} must be a  fixed point of this operator. By constructing a suitable initial value for which the associated Lyapunov--Perron operator has no nontrivial fixed point, we show the existence of an unbounded solution which leads to the instability of the system.

%Note that this technique is also used to establish stable manifold %for nonlinear fractional differential equations whose linear parts %are hyperbolic, i.e.,,,,,, the spectrum of the linear parts
%do not intersect the line $\{z\in \C: \hbox{arg}%(z)=\frac{\alpha\pi}{2}\}$, see . In comparison to the %work in

We note that \cite{LiMa13} claimed that they proved the conjecture above, but their paper contains serious flaws which make their construction as well as their proof incorrect; a discussion about their paper will be given in Remark~\ref{discussion} below.

The paper is organized as follows. Section 2 is a preparatory section where we present some basic notions from fractional calculus and give some basic properties of Mittag-Leffler functions. Section 3 is devoted to the main result of the paper in which we prove a theorem on instability of the trivial solution of the nonlinear Caputo fractional differential system \eqref{IntroEq}.

To conclude the introductory section, we fix some notation which will be used later.
Let $\R_{\geq 0}$ denote the set of all nonnegative real numbers. For $\alpha \in (0,1)$, we define
\begin{equation}\label{Sector}
\Lambda_{\alpha}^u:=\left\{z\in\C\setminus\{0\}:|\text{arg}(z)|< \frac{\alpha \pi}{2}\right\}.
\end{equation}
Let $(X,\|\cdot\|)$ be a Banach space. Denote by
$C(\R_{\geq 0};X)$ the linear space of all continuous functions $\xi:\R_{\geq 0}\rightarrow X$, and by
$C_\infty(\R_{\geq 0};X)$ the linear space of all continuous functions $\xi:\R_{\geq 0}\rightarrow X$ such that
\[
\|\xi\|_\infty:=\sup_{t\in \R_{\geq 0}}\|\xi(t)\|<\infty.
\]
Clearly, $(C_\infty(\R_{\geq 0};X),\|\cdot\|_\infty)$ is a Banach space.
\section{Preliminaries}
\subsection{Fractional differential equations}
We briefly recall an abstract framework of fractional calculus.
%----------------------------
% Fractional calculus
%-----------------------------

Let $\alpha>0$, $[a,b]\subset \R$ and $x:[a,b]\rightarrow \R$ be a measurable function such that $\int_a^b|x(\tau)|\;d\tau<\infty$. The \emph{Riemann--Liouville integral operator of order $\alpha$} is defined by
\[
(I_{a+}^{\alpha}x)(t):=\frac{1}{\Gamma(\alpha)}\int_a^t(t-\tau)^{\alpha-1}x(\tau)\;d\tau,
\]
where $\Gamma(\cdot)$ is the Gamma function.
The \emph{Caputo fractional derivative} $^{C\!}D_{a+}^\alpha x$ of a function $x\in C^m([a,b])$ is defined by
\[
(^{C\!}D_{a+}^\alpha x)(t):=(I_{a+}^{m-\alpha}D^mx)(t),
\]
where $D=\frac{d}{dt}$ is the usual derivative and $m:=\lceil\alpha\rceil$ is the smallest integer larger or equal to $\alpha$.

The Caputo fractional derivative of a $d$-dimensional vector function $x(t)=(x_1(t),\dots,x_d(t))^{\rT}$ is defined component-wise as
\[
(^{C\!}D_{a+}^\alpha x)(t):=(^{C\!}D_{a+}^\alpha x_1(t),\dots,^{C\!}D_{a+}^\alpha x_d(t))^{\rT}.
\]

We now recall the notions of stability of the trivial solution of the fractional differential equation \eqref{IntroEq}, see also \cite[Definition 7.2, p.~157]{Kai}. Note that since $f$ is locally Lipschitz continuous, for any initial value $x_0\in\R^d$ in a neighborhood of 0, the equation \eqref{IntroEq} has a unique solution, which we denote by $\varphi(\cdot,x_0)$, with its maximal interval of existence $I = [0,t_{max}(x_0))$, $0<t_{max}(x_0)\leq \infty$.
\begin{definition}\label{DS}
The trivial solution of \eqref{IntroEq} is called \emph{stable} if for any $\varepsilon >0$ there exists $\delta=\delta(\varepsilon)>0$ such that for every $\|x_0\|<\delta$ we have $t_{max}(x_0) = \infty$ and
\[
\|\phi(t,x_0)\|\leq \eps\qquad\hbox{for } t\ge 0.
\]
The trivial solution is called \emph{unstable} if it is not stable.
\end{definition}
\subsection{Mittag-Leffler function}
The Mittag-Leffler function is a generalization of the exponential function. Like the exponential function plays a very important role in the theory of ordinary differential equation, the Mittag-Leffler function is at the heart of the theory of fractional differential equations (see, e.g., \cite{Podlubny}). The Mittag-Leffler function is defined by the
formula
\[
E_{\alpha,\beta}(z):=\sum_{k=0}^\infty \frac{z^k}{\Gamma(\alpha k+\beta)},\qquad E_{\alpha}(z):=E_{\alpha,1}(z),
\]
where $z\in \C$, $\alpha>0$, $\beta >0$. For a constant matrix $A$ the matrix-valued Mittag-Leffler function is defined by
\[
E_{\alpha,\beta}(A):=\sum_{k=0}^\infty \frac{A^k}{\Gamma(\alpha k+\beta)},\qquad E_{\alpha}(A):=E_{\alpha,1}(A).
\]
We present in this subsection some basic properties of Mittag-Leffler functions. These results are slight refinements of known results in the theory of Mittag-Leffler functions which are adapted to our case. To derive these estimates one uses the integral representation of Mittag-Leffler functions, see e.g., \cite{Podlubny}; we give only a sketch of the proof.

\begin{lemma}\label{Lemma3}
Let $\lambda\in\Lambda_\alpha^u$, where $\Lambda_\alpha^u$ is defined in \eqref{Sector}. There exist a real number $t_0>0$ and a positive constant $m(\alpha,\lambda)$ such that the following estimates hold:
\begin{align*}
\left|E_\alpha(\lambda t^{\alpha})-\frac{1}{\alpha}\exp{(\lambda^{\frac{1}{\alpha}}t)}\right|&\le \frac{m(\alpha,\lambda)}{t^{\alpha}},\\
\left|t^{\alpha-1}E_{\alpha,\alpha}(\lambda t^{\alpha})-\frac{1}{\alpha}\lambda^{\frac{1}{\alpha}-1
}\exp{(\lambda^{\frac{1}{\alpha}}t)}\right|&\le \frac{m(\alpha,\lambda)}{t^{\alpha+1}},
\end{align*}
for every $t\geq t_0$.
\end{lemma}

For a proof of this theorem one uses integral representations of Mittag-Leffler functions and estimates for the integrals similarly as in the proofs of Theorem 1.3 and Theorem 1.4 in \cite[pp.~32--34]{Podlubny}.

\begin{lemma}\label{Lemma4}
Let $\lambda\in\Lambda_\alpha^u$, where $\Lambda_\alpha^u$ is defined in \eqref{Sector}. There exists a positive constant $K(\alpha,\lambda)$ such that the following estimates hold:
\begin{align*}
&\int_t^\infty \left|\lambda^{\frac{1}{\alpha}-1}E_\alpha(\lambda t^\alpha)
\exp(-\lambda^{\frac{1}{\alpha}}\tau)g(\tau)\right|\;d\tau
\leq K(\alpha,\lambda)\|g\|_\infty,\\
&\int_0^t\left|
\left((t-\tau)^{\alpha-1}E_{\alpha,\alpha}(\lambda(t-\tau)^{\alpha})- \lambda^{\frac{1}{\alpha}-1}E_\alpha(\lambda t^\alpha)\exp(-\lambda^{\frac{1}{\alpha}}\tau)\right)g(\tau)\right|\;d\tau\\
&\hspace{6.5cm} \leq K(\alpha,\lambda)\|g\|_\infty,
\end{align*}
for all $t \ge 0$ and any function $g\in C_\infty(\R_{\ge 0};\C)$.
\end{lemma}
\begin{proof}
The proof of this lemma follows easily by using Lemma \ref{Lemma3} and repeating arguments used in the proofs of Lemma 5 and Lemma 6 in \cite{Cong}.
\end{proof}

\begin{lemma}\label{LimitLemma}
For any function  $g\in C_\infty(\R_{\ge 0};\C)$ and $\lambda\in\Lambda_{\alpha}^u$, we have the following limiting relations:
\begin{align}\label{Eq4}
\lim_{t\to\infty}&\int_0^t
\notag (t-\tau)^{\alpha-1}\frac{E_{\alpha,\alpha}(\lambda(t-\tau)^\alpha)}{E_\alpha(\lambda t^\alpha)}g(\tau)\;d\tau\\
&\hspace{0.5cm}=\; \lambda^{\frac{1}{\alpha}-1}\int_0^\infty\exp(-\lambda^{\frac{1}{\alpha}}\tau)g(\tau)\;d\tau.
\end{align}
\end{lemma}
\begin{proof}
Use Lemma \ref{Lemma3}, Lemma \ref{Lemma4}, and arguments similar to that of the proof of Lemma 8 in \cite{Cong}.
\end{proof}

\section{Instability of Fractional Differential Equations}
We now state the main result of this paper about a criterion of instability of fractional differential equation.
\begin{theorem}\label{main result 1}
Consider the nonlinear fractional differential equation
\begin{equation}\label{mainEq}
^{C\!}D_{0+}^\alpha x(t)=Ax(t)+f(x(t)),
\end{equation}
where $t\in \R_{\ge 0}$, $x\in\R^d$, $A\in \R^{d\times d}$, and $f:\R^d \rightarrow \R^d$ is Lipschitz continuous on a neighborhood of the origin. Assume that the spectrum $\sigma(A)$ of $A$ satisfies
\begin{equation}\label{SpecCond}
\sigma(A)\cap  \Lambda_{\alpha}^u\ne \emptyset,
\end{equation}
and $f$ satisfies
\begin{equation}\label{LipschitzCond}
f(0)=0\quad\hbox{and}\quad\lim_{r\to 0} \ell_f(r)=0.
\end{equation}
Then, the trivial solution of \eqref{mainEq} is unstable.
\end{theorem}

For a proof of this theorem we follow the approach of \cite{Cong_3,Cong_4}. Namely, first we transform the linear part to a simple form; then construct an appropriate Lyapunov--Perron operator which is a contraction and its fixed point is a solution of \eqref{mainEq}, and exploit the property of the Lyapunov--Perron operator to derive the conclusion of the theorem. To do this we need some preparatory steps, the details of which we now present.

{\bf Transformation of the linear part}

Let $\sigma(A):=\{\hat\lambda_1,\dots, \hat\lambda_m\}$ denote the spectrum of $A$, i.e.,\ $\hat\lambda_1,\dots,\hat\lambda_m \in \C$ is the collection of all the distinct eigenvalues of $A$. Let $T\in\C^{d\times d}$ be a nonsingular matrix  transforming $A$ into its Jordan normal form, i.e.,
\[
T^{-1}A T=\hbox{diag}(A_1,\dots,A_n),
\]
where for $i=1,\dots,n$ the block $A_i$ is of the following form
\[
A_i=\lambda_i\, \id_{d_i\times d_i}+\eta_i\, N_{d_i\times d_i},
\]
where $\eta_i\in\{0,1\}$, $\lambda_i \in \sigma(A)$, and the nilpotent matrix $N_{d_i\times d_i}$ is given by
\[
N_{d_i\times d_i}:=
\left(
      \begin{array}{*7{c}}
      0  &     1         &    0      & \cdots        &  0        \\
        0        & 0    &    1     &   \cdots      &              0\\
        \vdots &\vdots        &  \ddots         &          \ddots &\vdots\\
        0 &    0           &\cdots           &  0 &          1 \\

        0& 0  &\cdots                                          &0         & 0 \\
      \end{array}
    \right)_{d_i \times d_i}.
\]
Note that with this transformation we leave the field of real numbers and consider differential equations in the complex numbers. Only if all eigenvalues of $A$ are real, we remain in $\R$. For a general real-valued matrix $A$ we may simply embed $\R$ into $\C$, consider $A$ as a complex-valued matrix and thus get the above Jordan form for $A$. Alternatively, we may use a real-valued Jordan form  (see \cite[Chapter 6, p.\ 243]{Lancaster}; for a discussion on similar issues for FDE see also \cite[pp.\ 152--153]{Kai}). For simplicity we use the embedding method and omit the discussion on how to return back to the field of real numbers. Note also that those techniques are well known in the theory of ordinary differential equations.

Let $\gamma$ be an arbitrary but fixed positive number. Using the transformation $P_i:=\textup{diag}(1,\gamma,\dots,\gamma^{d_i-1})$, we obtain that
\begin{equation*}
P_i^{-1} A_i P_i=\lambda_i\, \id_{d_i\times d_i}+\gamma_i\, N_{d_i\times d_i},
\end{equation*}
$\gamma_i\in \{0,\gamma\}$. Put $P:= \textup{diag}(P_1,\dots,P_n)$, then under the transformation $y:=(TP)^{-1}x$ system \eqref{mainEq} becomes (denoting the solution again by $x$ instead of $y$)
\begin{align}\label{NewSystem}
^{C\!}D_{0+}^\alpha x(t)&=\hbox{diag}(J_1,\dots,J_n)x(t)+h(x(t)),\\
\notag &=Jx(t)+h(x(t)),
\end{align}
where $J:=\text{diag}(J_1,\dots,J_n)$, $J_i:=\lambda_i \id_{d_i\times d_i}$ for $i=1,\dots,n$ and the function $h$ is given by
\begin{equation}\label{Eq3}
h(x):=\text{diag}(\gamma_1N_{d_1\times d_1},\dots,\gamma_nN_{d_n\times d_n})x+(TP)^{-1}f(TPx).
\end{equation}
Note that
\begin{equation}\label{eqn.new-h}
h(0)=0,\qquad \lim_{r\to 0}\ell_h (r)= \left\{
\begin{array}{ll}
\gamma  & \hbox{if there exists } \gamma_i=\gamma,\\[1ex]
0 & \hbox{otherwise}.
\end{array}
\right.
\end{equation}
Since the spectrum of $A$ satisfies the instability condition we can find at least one eigenvalue $\hat\lambda_i\in \Lambda_{\alpha}^u$. Without loss of generality, we can assume that $\hat\lambda_i\in\Lambda_{\alpha}^u$ for $i=1\ldots,k'$, $1 \leq k' \leq d$. Consequently, for simplicity of notation we can write \eqref{NewSystem} in the form
\begin{equation}\label{NewSystem1}
^{C}D_{0+}^\alpha x(t)=\hbox{diag}(\mu_1,\dots,\mu_d)x(t)+h(x(t)),
\end{equation}
where $h$ is defined by \eqref{Eq3} and
\begin{align}\label{neweq.instability}
\notag \mu_i  &\in\sigma(A) = \{\hat\lambda_1,\ldots,\hat\lambda_m\}, \qquad i=1,\ldots,d,\\
 \mu_i  &\in\Lambda_{\alpha}^u,\qquad\hspace*{3cm} i=1,\ldots,k,\\
\notag \mu_i  &\in\sigma(A)\setminus \Lambda_{\alpha}^u, \qquad\hspace*{1.8cm} i=k+1,\ldots,d.
\end{align}
%\end{remark}
%
\begin{remark}\label{Remark2}
Since the transforming matrix $TP$ is constant, the type of stability of the trivial solution of equations \eqref{mainEq} and \eqref{NewSystem1} are the same,
 i.e., they are either both stable or both unstable.
\end{remark}

{\bf Construction of an appropriate Lyapunov--Perron operator}

We define a specific Lyapunov--Perron operator associated with the equation \eqref{NewSystem1} as follows. For any $\xi \in C_{\infty}(\R_{\ge 0};\C^d)$, the Lyapunov--Perron operator $\mathcal{T}: C_{\infty}(\R_{\ge 0};\C^d) \rightarrow C(\R_{\ge 0};\C^d)$ is defined by
\begin{equation}\label{eqn.LyaPer}
\mathcal{T}\xi(t):=  ((\mathcal{T}\xi)^1(t),\dots,(\mathcal{T}\xi)^d(t))^{\rT}\qquad\hbox{for all}\quad t\in\R_{\geq 0},
\end{equation}
where for $i=1,\ldots,k$ we set
\begin{align}\label{DefineT1}
\notag (\mathcal{T}\xi)^i(t):=& \int_0^t (t-\tau)^{\alpha-1}E_{\alpha,\alpha}(\mu_i(t-\tau)^\alpha )h^i(\xi(\tau))\;d\tau\\
 &\hspace{0.5cm} -\mu_i^{\frac{1}{\alpha}-1}E_\alpha(\mu_it^\alpha)\int_0^\infty \exp{\big(-\mu_i^{\frac{1}{\alpha}}\tau\big)}h^i(\xi(\tau))\;d\tau,
\end{align}
 and for $i=k+1,\ldots,d$ we set
\begin{equation}\label{DefineT2}
(\mathcal{T}\xi)^i(t):=\int_0^t (t-\tau)^{\alpha-1}E_{\alpha,\alpha}(\mu_i(t-\tau)^\alpha)h^i(\xi(\tau))\;d\tau,
\end{equation}
here $h(x) = (h^1(x),\ldots,h^d(x))^{\rT}$ is the coordinate representation of the vector $h(x)$.
%\begin{remark}
Note that under an additional assumption that $\mu_i\not=0$ and $|\mbox{arg}(\mu_i)|\not=\frac{\alpha \pi}{2}$ for $i=1,\dots,d$,  it is proved in \cite{Cong_4} that there exists a neighborhood of the trivial function in $C_{\infty}(\R_{\geq 0};\C^d)$ equipped with the sup norm such that the operator $\mathcal T$ is a contraction on this neighborhood. 
 In this paper, it is only assumed  that $|\mbox{arg}(\mu_i)|< \frac{\alpha \pi}{2}$ for $i=1,\dots,k$ and for $i=k+1,\dots,d$ either $\mu_i=0$ or $|\mbox{arg}(\mu_i)|\geq \frac{\alpha \pi}{2}$. Therefore, the operator $\mathcal T$ is in general not a contraction with respect to the sup norm. To see this, consider the case that $d=2, k=1$, $\mu_2=0$, $h(x):=(0,x_2^2)^{\rT}$, then the operator $\mathcal T$ defined by \eqref{eqn.LyaPer} has its second coordinate given by
\[
(\mathcal{T}\xi)^2(t):=\int_0^t (t-\tau)^{\alpha-1} h^2(\xi(\tau))\;d\tau,
\]
hence, obviously, $\mathcal T$ as well as any of its restrictions to small balls around the origin in $C_{\infty}(\R_{\ge 0};\C^2)$ is not a contraction.
%Hence, for  $h(x):=(0,x_2^2)^{\rT}$ we have $\mathcal T C_{\infty}(\R_{\ge 0};\C^2)\not\subset  C_{\infty}(\R_{\ge 0};\C^2)$.
%\end{remark}

In what follows,  by introducing a suitable weighted norm $\|\cdot\|_w$ on $C_\infty(\R_{\geq 0};\C^d)$, we show that the Lyapunov--Perron operator $\mathcal T$ is contractive on a neighborhood of the trivial function.

For $i=1,\ldots,k$, we write $\mu_i$ in the form:
\[
\mu_i=r_i(\cos \phi_i+\iota \sin \phi_i),
\]
where $\iota$ denotes the imaginary unit, $r_i>0$ is the modulus and $-\frac{\alpha \pi}{2}<\phi_i<\frac{\alpha \pi}{2}$ is the argument of the complex number $\mu_i\in \Lambda_\alpha^u$. Let
\begin{equation}\label{WeightFactor}
w:=\frac{\min_{1\le i\le k}\big\{r_i^{\frac{1}{\alpha}}\cos \frac{\phi_i}{\alpha}\big\}}{3} >0,
\end{equation}
and define a weighted norm-like function on $C(\R_{\ge 0};\C^d)$ by
\[
\|\xi\|_w:=\sup_{t\ge 0}\frac{\|\xi(t)\|}{\exp(w t)}.
\]
It is easily seen that $\|\cdot\|_w$ has the three properties of a norm but it might take the value $+\infty$ on the space $C(\R_{\ge 0};\C^d)$. It is a norm on the space $C_\infty(\R_{\geq 0};\C^d)$ and $(C_\infty(\R_{\geq 0};\C^d), \|\cdot\|_w)$ is a Banach space.

From \eqref{Eq3} and \eqref{eqn.new-h} it follows that $h$ is Lipschitz continuous on a neighborhood of the origin. Hence, there exists $\widehat\varepsilon>0$ such that $\ell_h(r)<\infty$ for all $0<r<\widehat\varepsilon$, where  $\ell_h(r)$ is defined according to \eqref{eqn.l_f}.  Using the notation 
\[
B_{C_\infty}(0,\widehat\varepsilon):=\Big\{\xi\in C(\R_{\ge 0};\C^d):\left||\xi|\right|_\infty\le \widehat\varepsilon\Big\},
\]
the following proposition gives an estimate for the operator $\mathcal{T}$.
\begin{proposition}\label{Prp2}
Consider system \eqref{NewSystem1}. Then, there exists a positive constant $K(\alpha,A)$ depending on $\sigma(A)$ such that the following inequality holds
\[
 \|\mathcal {T}\xi-\mathcal{T}\hat{\xi}\|_w  \le K(\alpha,A)\; \ell_h (\max\{\|\xi\|_\infty,\|\hat{\xi}\|_\infty\})\|\xi-\hat{\xi}\|_w,
\]
for all $\xi,\hat{\xi}\in B_{C_\infty}(0,\widehat\varepsilon)$.
\end{proposition}
\begin{proof}
Let $\xi,\hat{\xi}\in C_\infty(\R_{\ge 0};\C^d)$ be arbitrary. For  any index $i\in \{1,\ldots,k\}$ and any $t\geq 0$, from the definition \eqref{DefineT1} of the operator $\mathcal{T}$, we have
\begin{align*}
&|(\mathcal{T}\xi)^i(t)-(\mathcal{T}\hat{\xi})^i(t)|\\
&\hspace{1cm}\le \int_0^t \left|(t-\tau)^{\alpha-1}E_{\alpha,\alpha}(\mu_i(t-\tau)^\alpha)-\mu_i^{\frac{1}{\alpha}-1}E_\alpha(\mu_it^\alpha)\exp{\big(-\mu_i^{\frac{1}{\alpha}}\tau\big)}\right|\\
&\hspace{2.0cm}\cdot |h^i(\xi(\tau))-h^i(\hat{\xi}(\tau))|\;d\tau\\
&\hspace{1.5cm}+\int_t^\infty \left|\mu_i^{\frac{1}{\alpha}-1}E_\alpha(\mu_it^\alpha) \exp{\big(-\mu_i^{\frac{1}{\alpha}}\tau\big)}\right|\cdot |h^i(\xi(\tau))-h^i(\hat{\xi}(\tau))|\;d\tau.
\end{align*}
Hence,
\begin{align*}
&\frac{|(\mathcal{T}\xi)^i(t)-(\mathcal{T}\hat{\xi})^i(t)|}{\exp(w t)}\\
&\hspace{1cm}\le  \int_0^t \left|(t-\tau)^{\alpha-1}E_{\alpha,\alpha}(\mu_i(t-\tau)^\alpha)-\mu_i^{\frac{1}{\alpha}-1}E_\alpha(\mu_it^\alpha)\exp{\big(-\mu_i^{\frac{1}{\alpha}}\tau\big)}\right|\\
& \hspace{2.0cm}\cdot \ell_h(\max\{\|\xi\|_\infty,\|\hat{\xi}\|_\infty\})\frac{\|\xi(\tau)-\hat{\xi}(\tau)\|}{\exp(w \tau)}\;d\tau\\
&\hspace{1.5cm}+\int_t^\infty \left|\mu_i^{\frac{1}{\alpha}-1}E_\alpha(t^\alpha \mu_i) \exp{\big(-\mu_i^{\frac{1}{\alpha}}\tau\big)}\right|\exp{w (\tau-t)}\\
&\hspace{2.3cm}\cdot\ell_h(\max\{\|\xi\|_\infty,\|\hat{\xi}\|_\infty\})\frac{\|\xi(\tau)-\hat{\xi}(\tau)\|}{\exp(w \tau)}\;d\tau.
\end{align*}
Since $\mu_i\in\Lambda_{\alpha}^u$, according to Lemma \ref{Lemma4}, we can find a constant $K(\alpha,\mu_i)>0$ such that for all $t\geq 0$
\begin{equation}\label{est_1}
\int_0^t \left|(t-\tau)^{\alpha-1}E_{\alpha,\alpha}(\mu_i(t-\tau)^\alpha)-\mu_i^{\frac{1}{\alpha}-1}E_\alpha(\mu_it^\alpha)\exp{\big(-\mu_i^{\frac{1}{\alpha}}\tau\big)}\right|\,d\tau
\leq K(\alpha,\mu_i).
\end{equation}
Furthermore, by definition of $w$ as in \eqref{WeightFactor} and  Lemma~\ref{Lemma3} there exist $t_0>0$ and $m(\alpha,\mu_i)>0$ such that for any $t \geq t_0$ we have
\begin{align}\label{est_2}
\notag & \int_t^\infty \left|\mu_i^{\frac{1}{\alpha}-1}E_\alpha(\mu_it^\alpha) \exp{(-\mu_i^{\frac{1}{\alpha}}\tau)}\right|\exp(w (\tau-t))\;d\tau\\
\notag & \le r_i^{\frac{1}{\alpha}-1}\int_t^\infty \left(\frac{\exp\big(r_i^{\frac{1}{\alpha}}t\cos\frac{\phi_i}{\alpha}\big)}{\alpha}+\frac{m(\alpha,\mu_i)}{t_0^\alpha}\right) |\exp{(-\mu_i^{\frac{1}{\alpha}}\tau)}| \exp(w (\tau-t))\;d\tau\\
&\le r_i^{\frac{1}{\alpha}-1}\left(\frac{1}{\alpha}+\frac{m(\alpha,\mu_i)}{t_0^\alpha}\right)\int_0^\infty \exp(-w \tau)\;d\tau,
\end{align}
while for $0\le t\le t_0$
\begin{align}\label{est_3}
\notag &\int_t^\infty \left|\mu_i^{\frac{1}{\alpha}-1}E_\alpha(\mu_it^\alpha) \exp{\big(-\mu_i^{\frac{1}{\alpha}}\tau\big)}\right|\exp(w (\tau-t))\;d\tau\\
&\hspace{1cm} \le r_i^{\frac{1}{\alpha}-1} E_\alpha(r_i t_0^\alpha)\int_0^\infty \exp(-w \tau)\;d\tau.
\end{align}
From \eqref{est_1}, \eqref{est_2}, and \eqref{est_3}, for each $i=1,\ldots,k$, we can find a constant $\hat{K}(\alpha,\mu_i) > 0$ such that
\[
 \|(\mathcal {T}\xi)^i-(\mathcal{T}\hat{\xi})^i\|_w  \le \hat{K}(\alpha,\mu_i)\; \ell_h (\max\{\|\xi\|_\infty,\|\hat{\xi}\|_\infty\})\|\xi-\hat{\xi}\|_w.
\]
Now for $i=k+1,\ldots,d$, using \eqref{DefineT2}, we obtain for any $t\geq 0$ the estimate
\begin{align*}
 \frac{|(\mathcal{T}\xi)^i(t)-(\mathcal{T}\hat{\xi})^i(t)|}{\exp(w t)}\le &\int_0^t (t-\tau)^{\alpha-1}|E_{\alpha,\alpha}(\mu_i (t-\tau)^\alpha)|\exp(-w (t-\tau))\\
&\hspace{0.5cm} \cdot \ell_h(\max\{\|\xi\|_\infty,\|\hat{\xi}\|_\infty\})\frac{|\xi(\tau)-\hat{\xi}(\tau)|}{\exp(w \tau)}\;d\tau,
\end{align*}
which implies that
\begin{align*}
 \frac{|(\mathcal{T}\xi)^i(t)-(\mathcal{T}\hat{\xi})^i(t)|}{\exp(w t)}\le &\int_0^t \tau^{\alpha-1}|E_{\alpha,\alpha}(\mu_i \tau^\alpha)|\exp(-w \tau)\;d\tau\\
&\hspace{0.5cm} \cdot \ell_h(\max\{\|\xi\|_\infty,\|\hat{\xi}\|_\infty\})\|\xi(\tau)-\hat{\xi}(\tau)\|_w.
\end{align*}
Note that for $\mu_i\in\sigma(A)\setminus \Lambda_\alpha^u$, it is clear that
$|E_{\alpha,\alpha}(\mu_i t^\alpha)|$ is bounded on $\R_{\ge 0}$, and thus
\[
\int_0^\infty \tau^{\alpha-1}|E_{\alpha,\alpha}(\mu_i \tau^\alpha)|\exp(-w \tau)\;d\tau<\infty.
\]
Hence, for $i=k+1,\ldots, d$, the following inequalities hold
\begin{align*}
& \|(\mathcal {T}\xi)^i-(\mathcal{T}\hat{\xi})^i\|_w\\
 &\hspace{1.0cm} \le \int_0^\infty \tau^{\alpha-1}|E_{\alpha,\alpha}(\mu_i \tau^\alpha)|\exp(-w \tau)d\tau \cdot \ell_h (\max\{\|\xi\|_\infty,\|\hat{\xi}\|_\infty\})\|\xi-\hat{\xi}\|_w.
\end{align*}
Let $\hat{K}(\alpha,\Lambda_\alpha^u):=\max_{1\le i\le k}\hat{K}(\alpha,\mu_i)$ and
$$
 K(\alpha,A):=\max\left\{\hat{K}(\alpha,\Lambda_\alpha^u),\max_{k+1\le i\le d}\int_0^\infty \tau^{\alpha-1}|E_{\alpha,\alpha}(\mu_i \tau^\alpha)|\exp(-w\tau)\;d\tau \right\},
$$
then the proof is complete.
\end{proof}
From now on, we choose and fix the parameter $\gamma>0$ in the definition of the transformation $P$ above such that $\gamma<\frac{1}{2K(\alpha,A)}$. By this choice of $\gamma$ the
 systems \eqref{NewSystem} and \eqref{NewSystem1} are completely specified. Due to \eqref{eqn.new-h}, there exists a positive constant $\varepsilon$ such that 
\begin{equation}\label{Eq7a}
0 < \varepsilon < \widehat\varepsilon, \qquad\hbox{and}\qquad
K(\alpha,A)\;\ell_h (\varepsilon) \leq \frac{2}{3}.
\end{equation}
By using \eqref{Eq7a} and Proposition \ref{Prp2}, we obtain immediately the following property of $\mathcal T$.
\begin{proposition}\label{Lemma6}
For all  $\xi,\hat{\xi}\in B_{C_\infty}(0,\varepsilon)$, where $\varepsilon$ is a positive number satisfying \eqref{Eq7a},  we have
\begin{equation*}
\|\mathcal{T}\xi-\mathcal{T}\hat{\xi}\|_w
\leq
\frac{2}{3}\;\|\xi-\widehat{\xi}\|_w.
\end{equation*}
\end{proposition}
\begin{proof}[Proof of Theorem \ref{main result 1}]
Due to Remark \ref{Remark2}, it is sufficient to prove the instability for the trivial solution of system \eqref{NewSystem1}. Assume to the contrary that the trivial solution of \eqref{NewSystem} is stable. Let $\varepsilon>0$ satisfy \eqref{Eq7a} and assume that the function $f$ is Lipschitz continuous on $B_{\C^d}(0,\eps)$. Choose $\delta=\delta(\varepsilon)>0$ such that for any $x_0\in B_{\C^d}(0,\delta)$ the solution $\phi(\cdot,x_0)$ satisfies $\|\phi(t,x_0)\|\le \varepsilon$ for every $t\ge 0$. Note that a vector $x\in \C^d$ can be written component-wise as $x=(x^1,x^2,\ldots,x^d)^{\rm T}$. Take and fix a vector $x_0=(x_0^1,\ldots,x_0^k,0,\ldots,0)^{\rm T}\in B_{\C^d}(0,\delta)\setminus\{0\}$. 
By  the variation of constants formula  for fractional differential equations  (see \cite[Theorem 1]{Cong} and \cite[Remark 7.1, pp. 135--136]{Kai}), $\phi(\cdot,x_0)$ is represented in the form
\begin{align*}
\phi(t,x_0)&=E_{\alpha}(t^{\alpha}\hbox{diag}(\mu_1,\dots,\mu_d))x_0\\
& \hspace{1.0cm}+\int_0^t (t-\tau)^{\alpha-1}E_{\alpha,\alpha}((t-\tau)^{\alpha}J)h (\phi(\tau,x_0))\;d\tau\\
&=: (\phi^1(t,x_0),\phi^2(t,x_0), \ldots,\phi^d(t,x_0))^{\rm T}.
\end{align*}
Therefore, writing this formula component-wise, for $i=1,\ldots,k$, we have
\begin{align*}
\phi^i(t,x_0)&=E_\alpha(\mu_it^\alpha)x_0^i+\int_0^t (t-\tau)^{\alpha-1}E_{\alpha,\alpha}(\mu_i(t-\tau)^\alpha)h^i(\phi(\tau,x_0))\;d\tau\\
&=E_\alpha(\mu_it^\alpha)\left[x^i_0+\int_0^t \frac{(t-\tau)^{\alpha-1}E_{\alpha,\alpha}(\mu_i(t-\tau)^\alpha)h^i(\phi(\tau,x_0))}{E_\alpha (\mu_it^\alpha)}\;d\tau\right]
\end{align*}
and for $i=k+1,\ldots, d$, we have
$$
\phi^i(t,x_0)=\int_0^t (t-\tau)^{\alpha-1}E_{\alpha,\alpha}(\mu_i(t-\tau)^\alpha)h^i(\phi(\tau,x_0))\;d\tau.
$$
By virtue of Lemma \ref{Lemma3}, 
$\lim_{t\to\infty}|E_\alpha(\mu_i t^\alpha)|=\infty$  for  $i=1,\ldots, k$. Therefore, $\sup_{t\in\R_{\geq 0}}\|\phi(t,x_0)\|\leq \eps$ implies that for $i=1,\dots,k$
\begin{eqnarray*}
 x_0^i
 &=&
 -\lim_{t\to\infty} \int_0^t \frac{(t-\tau)^{\alpha-1}E_{\alpha,\alpha}(\mu_i(t-\tau)^\alpha)h^i(\phi(\tau,x_0))}{E_\alpha (\mu_it^\alpha)}\;d\tau\\
&=&
-\mu_i^{\frac{1}{\alpha}-1}\int_0^\infty \exp\big(-\mu_i^{\frac{1}{\alpha}}\tau)h^i(\phi(\tau,x_0))\;d\tau,
\end{eqnarray*}
where we use Lemma \ref{LimitLemma} to have the preceding limit. 
Therefore, $\phi(\cdot,x_0)$ is a fixed point of the operator $\mathcal{T}$, i.e.,
$\mathcal{T} \phi(\cdot,x_0) = \phi(\cdot,x_0)$.
Obviously $\mathcal{T}0=0$, and both the function $\phi(\cdot,x_0)$  and the trivial function $0$ belong to
  $B_{C_{\infty}}(0,\varepsilon)$.
Therefore, by virtue of Proposition~\ref{Lemma6} we have
\begin{align*}
\|\phi(\cdot,x_0)- 0\|_w&= \|\mathcal {T}\phi(\cdot,x_0)-\mathcal{T}0\|_w  \le K(\alpha;A)\; \ell_h (\varepsilon)\|\phi(\cdot,x_0)- 0\|_w\\
&\le \frac{2}{3}\;\|\phi(\cdot,x_0)-0\|_w,
\end{align*}
which implies that $\phi(t,x_0)=0$ for all $t\ge 0$. We arrive at  a contraction because $\phi(0,x_0)=x_0\ne 0$. Thus, the trivial solution of \eqref{NewSystem} is unstable and the proof of the theorem is complete.
\end{proof}

%\begin{corollary}
%Theorem \ref{main result} still hold if the conditions that $f$ is Lipschitz continuous in a neighborhood of the origin and $\lim_{r\to %0}\ell_f(0)=0$ are replaced by the conditions that $f$ is $C^1$ in a neighborhood of the origin and $\|Df(0)\|$ is small enough (the threshold %is determined by $A$).
%\end{corollary}
\begin{remark}
(i) We get the second part of the conjecture \cite[Theorem 3, p.~81]{Matignon} formulated in the beginning of this paper as a special case of Theorem~\ref{main result 1} when $d=1$.

(ii) Note that if the matrix $A$ is hyperbolic then the stability of the perturbed system \eqref{mainEq} can be determined by using the result of \cite{Cong_4}. If  $A$ is not hyperbolic, i.e.\ $A$ has an eigenvalue $\lambda$ such that $\lambda=0$ or $|\arg(\lambda)|=\frac{\pi\alpha}{2}$, then our Theorem~\ref{main result 1} show the instability of \eqref{mainEq} provided that 
$\sigma(A)\cap \Lambda_\alpha^u\not= \emptyset$. Otherwise, in case $A$ is non-hyperbolic and $\sigma(A)\cap \Lambda_\alpha^u = \emptyset$ the stability type of the trivial solution of \eqref{mainEq} is undecided. For an example, we consider the following one-dimensional equation
\begin{equation}\label{ex_extr}
^{\!C}D^\alpha_{0+} x(t)=0.
\end{equation}
The trivial solution of \eqref{ex_extr} is stable. However, if we add the perturbation $x^2(t)$ to this equation, then the trivial solution of the perturbed equation
\[
^{\!C}D^\alpha_{0+} x(t)=x^2(t)
\]
is unstable, see \cite[Lemma 2, p.~629]{LiMa13}. On the other hand, for the perturbation $-x^3(t)$, the trivial solution of the perturbed equation
\[
^{\!C}D^\alpha_{0+} x(t)=-x^3(t)
\]
is asymptotically stable, see \cite[p.~550]{Lam}.
\end{remark}
\begin{remark}[Discussion about the paper \cite{LiMa13} by Li and Ma]\label{discussion}
As mentioned in the Introduction, Li and Ma claimed in \cite{LiMa13} that they proved the Audounet--Matignon--Montseny conjecture as a consequence of their Theorem 3 in \cite{LiMa13}. However, their paper contains flaws which make their proof incorrect. More precisely, Theorem 2 in \cite{LiMa13} concerning a construction of a fractional flow in the Caputo sense  is false. For a simple counterexample consider the one-dimensional fractional system
\begin{equation*}
\begin{cases}
^{C\!}D_{0+}^\alpha x(t)=x(t),\qquad \alpha \in (0,1),\\
x(0)=x_0,
\end{cases}
\end{equation*}
which can be solved explicitly and the solution is $x(t) = x_0 E_\alpha(t^\alpha)$ (see \cite[Theorem 7.2, p.~135]{Kai}). For this system, using the notation of \cite{LiMa13}, for $s=1$ and $t > 0$ we have
\begin{align*}
\theta_t\circ \phi_1(x_0)&= x_0 + \frac{1}{\Gamma(\alpha)}\int_0^1(t+1-\tau)^{\alpha-1} E_\alpha(\tau^\alpha)x_0 d\tau\\
&=x_0\Big(1+\frac{1}{\Gamma(\alpha)}\int_0^1(t+1-\tau)^{\alpha-1} E_\alpha(\tau^\alpha)d\tau\Big).
\end{align*}
Hence
$$
\phi_t\circ \theta_t\circ \phi_1(x_0)
= x_0\Big(1+\frac{1}{\Gamma(\alpha)}\int_0^1(t+1-\tau)^{\alpha-1} E_\alpha(\tau^\alpha)d\tau\Big)E_\alpha(t^\alpha).
$$
On the other hand,
$$
\phi_{t+1}(x_0) = x_0 E_\alpha\big((t+1)^\alpha\big).
$$
Theorem 2 of \cite{LiMa13} asserts that $\phi_{t+1}=\phi_t\circ \theta_t\circ \phi_1$ which amounts to
\begin{align*}
E_\alpha\big((t+1)^\alpha\big)&=
\left(1+\frac{1}{\Gamma(\alpha)}\int_0^1(t+1-\tau)^{\alpha-1} E_\alpha(\tau^\alpha)d\tau\right)E_\alpha(t^\alpha).
\end{align*}
However this assertion is false, because by using the asymptotic
behavior of the function $E_\alpha(\cdot)$ (see \cite[Theorem 1.3]{Podlubny} and Lemma~\ref{Lemma3} above),  for $t>0$ big enough one easily gets
\begin{align*}
E_\alpha\big((t+1)^\alpha\big)& >\left(1+\frac{E_\alpha(1)}{t^{1-\alpha}\Gamma(\alpha)}\right)E_\alpha(t^\alpha)\\
& >\Big(1+\frac{1}{\Gamma(\alpha)}\int_0^1(t+1-\tau)^{\alpha-1} E_\alpha(\tau^\alpha)d\tau\Big)E_\alpha(t^\alpha).
\end{align*}
 Thus, Theorem 2 of \cite{LiMa13} is false. The construction of a fractional flow in the Caputo sense in \cite{LiMa13} is therefore invalid. Moreover, the proof of Theorem 3 of \cite{LiMa13} which relies on Theorem 2 of \cite{LiMa13} is incorrect.
\end{remark}
\section*{Acknowledgement}
The work of the first, second and fourth author is supported by the Vietnam National Foundation for
Science and Technology Development (NAFOSTED) under Grant Number 101.03-2014.42. The third author is supported by the German Research Foundation (DFG) through the Cluster of Excellence 'Center for Advancing Electronics Dresden' (cfaed). The final part of this work was completed when the second and fourth author were visiting Vietnam Institute for Advanced Study in Mathematics (VIASM). They would like to thank VIASM for financial support of this visit.

\end{document}